\documentclass[12pt,leqno]{article}

\usepackage[cp1250]{inputenc}
\usepackage[OT4]{fontenc}
\usepackage{amsthm}
\usepackage{amsfonts}
\usepackage{amsmath}
\usepackage{amssymb}

\newcommand{\Q}{\mathbb{Q}}
\newcommand{\R}{\mathbb{R}}
\newcommand{\rank}{\operatorname{rank}}

\newtheorem{theorem} {Theorem}[section]
\newtheorem{theorem*}{Theorem}
\newtheorem{prop*} {Proposition}
\newtheorem{lemma*}{Lemma}
\newtheorem{lemma}[theorem]{Lemma}
\newtheorem{cor}[theorem]{Corollary}
\newtheorem{cor*}{Corollary}
\newtheorem{prop}[theorem] {Proposition}

\newtheorem{definition*}{Definition}
\newtheorem{rem}[theorem]{Remark}
\newtheorem{ex}[theorem]{Example}

\theoremstyle{definition}

\begin{document}

\begin{center}
{\bf   On the stable set of an analytic gradient flow}\\[1em]
by Zbigniew Szafraniec
\end{center}
{\bf Abstract.} Let $f:\R^n\rightarrow\R$, $n\geq 2$, be a real analytic function. In this paper 
we study the stable set of the gradient flow $\dot{x}=\nabla f(x)$ associated with
a critical point of $f$. In particular we present simple topological conditions 
which imply that this set contains an infinite family of trajectories,  or has a non-empty interior.

\section{Introduction.}
Let $f:\R^n\rightarrow\R$, $n\geq 2$, be an analytic function. According to {\L}ojasiewicz \cite{lojasiewicz},
the limit set of a trajectory of the dynamical system $\dot{x}=\nabla f(x)$ is either empty or contains
a single critical point of $f$. So the family of integral curves which converge to a critical point is a natural object of study
in the theory of gradient dynamical systems.

Let $f:\R^n,0\rightarrow\R,0$ be an analytic function defined in a neighbourhood of the origin, having a critical point at $0$.
We shall write $T(f)$ for the set of integral curves which converge to the origin, and
$S(f)$ for the stable set of the origin, which is the union of all orbits of the solutions
that converge to the origin. By \cite{lojasiewicz}, the stable set in closed
near the origin.

In this paper we study the naturally occurring question whether the set $T(f)$ is infinite or
whether the interior of $S(f)$ is non-empty? (In the planar case these problems are equivalent.)
By Remark \ref{plik_2_1_1}, if $T(f)$ is infinite then it has the cardinality of the continuum.

In some cases the answer is rather obvious. If the hessian matrix of $f$ at the origin has at least
two negative eigenvalues then the dimension of the stable manifold at the origin is $\geq 2$,
and then $T(f)$ is infinite. If the origin is a local strict maximum then $\operatorname{int}S(f)\neq\emptyset$.
It is worth pointing out that according to Moussu \cite[Theorem 3]{moussu} the family $T(f)$ always contain
trajectories which are represented by real analytic curves converging to the origin. In some cases
the family of those analytic curves can be infinite.

Let $\omega:\R^n,0\rightarrow\R,0$ be the homogeneous initial form associated with $f$.
Put $\Omega=S^{n-1}\cap\{\omega<0\}$. 
Applying the Moussu results \cite{moussu}  one may 
show that $\dim\operatorname\,S(f)\geq 2$ if there exists at least one
non-degenerate critical point of $\omega|\Omega$ which is not a local minimum.
In that case the set $T(f)$ is infinite.
Moreover,  if  there exists at least one
non-degenerate critical point of $\omega|\Omega$ which is  a local maximum
then $\operatorname{int}S(f)\neq\emptyset$.

 Let $S_r=S^{n-1}_r\cap\{f<0\}$, where $S^{n-1}_r=\{x\in\R^n\ |\ |x|=r\}$, $0<r\ll 1$.
By \cite{nowelszafraniec1}, \cite{nowelszafraniec2}, if $T(f)$ is finite then each  cohomology group $H^i(S_r)$ is trivial for  $i\geq 1$.
Hence, if there exists $i\geq 1$ with $H^i(S_r)\neq 0$, then $T(f)$ is infinite.

However, there are examples  where $n=2$ and none of the above assumptions holds but $T(f)$ is infinite (see Example \ref{plik_8_3} ).

In the course of proving Thom's gradient conjecture Kurdyka {\it et al.} \cite{kurdykaetal} have applied advanced techniques
of analytic geometry so as to investigate geometric properties of trajectories converging to a the origin. In particular
they have proved that with each such a trajectory one can associate a pair $(\ell,a)\in L'(f)$, where $\ell$ is a characteristic
exponent of $f$, the number $a$ is an asymptotic critical value of $f/|x|^{\ell}$, and $L'(f)$
is a finite subset of $\Q^+\times \R_-$, where $\Q^+$ is the set of positive rationals and
$\R_-$ is the set of negative real numbers.

In \cite{dzedzejszafraniec} (see also Section \ref{rozdzial_22}) there is presented an intrinsic  filtration of $T(f)$ given  in terms 
of characteristic exponents and asymptotic critical values of $f$. Unfortunately, these numbers are difficult
to compute. This is why in this paper we present methods which are more easy to apply.

The first main result of this paper shows that $T(f)$ is infinite if
$\rank\, H^0(S_r)< \rank\, H^0(\Omega)$ (see Theorem \ref{plik_8_1}),
i.e. if $S_r$ has less
connected components than $\Omega$.
As a corollary we shall show that the inequality $\chi(S_r)<\chi(\Omega)$ implies that $T(f)$ is infinite.
It is proper to add that there exist efficient methods of computing those Euler-Poincar\'e characteristics (see \cite{leckiszafraniec1}, \cite{szafraniec12}).
(These results have been earlier presented  in \cite{szafraniec13}.)

The second main result of this paper shows that $\operatorname{int}S(f)\neq\emptyset$  if
$\operatorname{rank} H^{n-2}(S_r)< \operatorname{rank} H^{n-2}(\Omega)$,
where $H^{n-2}(\cdot)$ is the $(n-2)$-th cohomology group with rational coefficients (see Theorem \ref{teor_1}).

Let $\Omega'=S^{n-1}\cap \{\omega\geq 0\}=S^{n-1}\setminus\Omega$, and $S_r'=S_r^{n-1}\cap \{f\geq 0\}=S_r^{n-1}\setminus S_r$, $0<r\ll 1$.
Sets $\Omega'$, $S_r'$ are compact and semianalytic, hence they are triagulable. By the Alexander duality theorem,
if $S_r'$ and $\Omega'$ are non-empty then
$\operatorname{rank}H_0(S_r')=1+\operatorname{rank}H^{n-2}(S_r)$ and
$\operatorname{rank}H_0(\Omega')=1+\operatorname{rank}H^{n-2}(\Omega)$.
Thus, if $S_r'$ has less connected components than $\Omega'$ then the interior of $S(f)$ is non-empty (see Theorem \ref{theor_2}).

Let $f$ be as above.
Assume that $g:\R^n,0\rightarrow\R,0$ an analytic function which is right-equivalent
to $f$.
We shall prove that $T(g)$ is infinite (resp. $\operatorname{int}S(g)\neq\emptyset$)
if $S_r$ has less connected components than $\Omega$ (resp. if $S_r'$ has less
connected components than $\Omega'$) (see Theorem \ref{cor_4_1}).

The paper is organized as follows. In Section \ref{rozdzial_01} we prove sufficient conditions which imply that a compact subset of the sphere has a non-empty interior. In Section \ref{rozdzial_1}
we study the homotopy type of some semi-analytic sets. In Section \ref{rozdzial_22}
we present properties of important geometric invariants
associated with trajectories of the gradient flow.
In Sections \ref{rozdzial_cardinality}, \ref{rozdzial_2}, we prove
the main results (Theorems \ref{plik_8_1}, \ref{teor_1}, \ref{theor_2}). 
Section \ref{rozdzial_3} is devoted to functions right-equivalent to the ones that satisfy assumptions of those theorems.
References \cite{absilkurdyka, bohmetal,  goldstein,  kurdykaetal, kurdykaparusinski, lageman, nowelszafraniec2} present  significant related results and
applications. 

\section{Sets with non-empty interior.}\label{rozdzial_01}

In this section we present some consequences of the Alexander duality theorem. The best reference here is \cite{spanier}.

\begin{lemma}\label{lem_1}
Suppose that $L\subset K$ are closed subsets of $S^{n-1}$, $n\geq 2$, and
$\operatorname{rank}\bar{H}^{n-2}(K)<\operatorname{rank}\bar{H}^{n-2}(L)<\infty$,
where $\bar{H}^{n-2}(\cdot)$ is the $(n-2)$-th \v{C}ech-Alexander-Spanier cohomology group with rational coefficients.
Then the interior of $K$ is non-empty.
\end{lemma}

\begin{proof} As $\bar{H}^{n-2}(L)\neq 0$ then sets $L$, $K$, $S^{n-1}\setminus L$ are not void.
If $K=S^{n-1}$ then the assertion holds. From now on we assume that $S^{n-1}\setminus K\neq\emptyset$ and $n\geq 3$.

By the Alexander duality theorem there are isomorphisms
$$\bar{H}^{n-2}(L)\simeq \tilde{H}_0(S^{n-1}\setminus L)\ , \ \bar{H}^{n-2}(K)\simeq \tilde{H}_0(S^{n-1}\setminus K),$$
where $\tilde{H}_0(\cdot)$ is the $0$-th reduced singular homology group with rational coefficients. 

Then $S^{n-1}\setminus L$ is a disjoint union of open connected components $U_1,\ldots,U_\ell$, 
where $\ell=1+\operatorname{rank}\tilde{H}_0(S^{n-1}\setminus L)=1+\operatorname{rank}\bar{H}^{n-2}(L)$, and
$S^{n-1}\setminus K$ is a disjoint union of open connected components $V_1,\cdots,V_k$,
where $k=1+\operatorname{rank}\tilde{H}_0(S^{n-1}\setminus K)=1+\operatorname{rank}\bar{H}^{n-2}(K)$.

Suppose that $U_i\setminus K\neq \emptyset$ for each $1\leq i\leq\ell$, so that
there are points $p_i\in U_i\setminus K$ and then $p_i\in V_{j(i)}$ for some $1\leq j(i)\leq k$.
As $V_{j(i)}$ is a connected subset of $U_1\cup\ldots\cup U_\ell$, then $V_{j(i)}\subset U_i$.

Because $U_i$ are pairwise disjoint, then  $V_{j(i)}$ are pairwise disjoint too. Hence $k\geq \ell$,
contrary to our claim.
Then at least one open connected component $U_i$ is a subset of $K$.

Similar arguments apply to the case where $n=2$.
\end{proof}

\begin{cor}\label{cor_1}
Suppose that $L\subset K\subset F$, where $L,K$ are compact, $n\geq 2$, 
$\operatorname{rank}\bar{H}^{n-2}(K)<\operatorname{rank}\bar{H}^{n-2}(L)<\infty$,
and $F$ is an $(n-1)$-dimensional manifold homeomorphic to a subset of $S^{n-1}$.
Then the interior of $K$ is non-empty.
\end{cor}

\section{Homotopy type of semi-analytic sets.}\label{rozdzial_1}

Let $f:\R^n,0\rightarrow\R,0$ be an analytic function defined in an open neighbourhood
of the origin. Let $\Q^+$ denote the set of positive rationals.
For $\ell\in\Q^+$, $a<0$, $y<0$ and $r>0$ we shall write
$$B_r^n=\{x\in\R^n\ |\ |x|\leq r\}\ ,\ \  S_r^{n-1}=\{x\in\R^n\ |\ |x|=r\},$$
$$V^{\ell,a}=\{x\in\R^n\setminus\{0\}\ |\ f(x)\leq a|x|^\ell\} ,\ $$
$$S_r^{\ell,a}=S_r^{n-1}\cap V^{\ell,a}\ ,\ \ B_r^{\ell,a}=B_r^n\cap V^{\ell,a},\ \ 
F^{\ell,a}(y)=f^{-1}(y)\cap V^{\ell,a},$$
$$D^{\ell,a}(y)= f^{-1}([y,0))\cap V^{\ell,a}=\{ x\in V^{\ell,a}\ |\ y\leq f(x)<0\}.$$

\begin{lemma} \label{plik_6_1} Assume that $\ell\in\Q^+$ and $a<0$.
If $0<-y\ll r\ll 1$ then
the sets $S_r^{\ell,a}$ and $F^{\ell,a}(y)$ are homotopy
equivalent. In particular, the singular cohomology groups $H^*(S_r^{\ell,a})$ and 
$H^*(F^{\ell,a}(y))$ are isomorphic.
\end{lemma}

\begin{proof} For $x\in  V^{\ell,a}\cup\{0\}$ lying sufficiently close to the origin we have $|x|^{1/2}\geq |f(x)|\geq |a|\cdot |x|^\ell$,
so that in particular functions $f(x)$, $|x|^2$  restricted to this set are proper.
According to the local triviality of proper analytic mappings between semi-analytic sets (see \cite{goreskymacpherson, hardt, hironaka}), they are locally trivial.
So there is $r_0>0$ such that $|x|:B_{r_0}^{\ell,a}\rightarrow (0,r_0]$ is a trivial fibration.
Hence the inclusion $S_r^{\ell,a}\subset B_r^{\ell,a}$ is a homotopy equivalence
for  each $0<r\leq  r_0$.

By similar arguments, there is $y_0<0$ such that $D^{\ell,a}(y_0)\subset B_{r_0}^{\ell,a}$ and
 $f:D^{\ell,a}(y_0)\rightarrow [y_0,0)$ is
a trivial fibration, so that the inclusion $F^{\ell,a}(y)\subset D^{\ell,a}(y)$
is a homotopy equivalence for each $y_0\leq y<0$.

 So, if $0<-y\ll r\ll 1$ then we may assume that $r\leq r_0$, $y_0\leq y$, and
$$D^{\ell,a}(y)\subset B_r^{\ell,a}\subset D^{\ell,a}(y_0)\subset B_{r_0}^{\ell,a}.$$
As inclusions $D^{\ell,a}(y)\subset D^{\ell,a}(y_0)$ and $B_r^{\ell,a}\subset B_{r_0}^{\ell,a}$
are homotopy equivalencies, then $D^{\ell,a}(y)\subset B_r^{\ell,a}$ is a homotopy
equivalence too. Then $F^{\ell,a}(y)$ is homotopy equivalent to $S_r^{\ell,a}$.
\end{proof}

 For $0<-y\ll r\ll 1$ we shall write
$$ F_r(y)=B_r^n\cap f^{-1}(y)\ ,\ \ S_r=\{x\in S_r^{n-1}\ |\ f(x)<0\}.$$
We call the set $F_r(y)$ the {\it real Milnor fibre}.
According to \cite{milnor}, it is either an $(n-1)$-dimensional compact manifold with boundary or an empty set.
Moreover, the sets $F_r(y)$ and $S_r$   are homotopy equivalent.

\begin{cor}\label{plik_6_2}
If $0<-y\ll r\ll 1$ then the cohomology groups $H^*(S_r)$ and $H^*(F_r(y))$ are isomorphic.
\end{cor}

\begin{lemma}\label{plik_6_2_1}
If $\ell\in\Q^+$, $a<0$ and $0<-y\ll r\ll 1$ then $F^{\ell,a}(y)=\{x\in F_r(y)\ |\ y\leq a|x|^\ell\}$.
In particular,  $F^{\ell,a}(y)\subset F_r(y)$.
\end{lemma}
\begin{proof}
If  $x\in F^{\ell,a}(y)$ then $y=f(x)\leq a|x|^\ell$, and then $|x|^\ell\leq y/a$.
If  $0<-y\ll r\ll 1$ then $|x|\leq r$, and then $x\in B_r^n\cap f^{-1}(y)=F_r(y)$.

If $x\in B_r^n\cap f^{-1}(y)$ and $y\leq a|x|^\ell$, then $x\in f^{-1}(y)\cap V^{\ell,a}=F^{\ell,a}(y)$.
Hence, if $0<-y\ll r\ll 1$ then $\{x\in F_r(y)\ |\ y\leq a|x|^\ell\}\subset F^{\ell,a}(y)$.
\end{proof}

Let $\omega$ be the initial form associated with $f$ and let $g=f-\omega$,
so that $f=\omega+g$. Denote by $d$ the degree of $\omega$. Hence $g=O(|x|^{d+1})$.

\begin{lemma}\label{plik_5_1} If $0<r\ll -a\ll 1$ then sets $S_r^{d,a}=S_r^{n-1}\cap\{f\leq a r^d\}$, 
$S_r^{n-1}\cap\{\omega\leq a r^d\}$ and $\Omega=S^{n-1}\cap\{\omega<0\}$
have the same homotopy type.
\end{lemma}

\begin{proof}
For $r\in\R$ sufficiently close to zero and $x\in S^{n-1}$ we have
$$f(rx)=\omega(rx)+g(rx)=r^d\omega(x)+r^{d+1}G(x,r),$$
where $G(x,r)$ is an analytic function defined in an open neighbourhood of $S^{n-1}\times\{0\}$.
Put $H(x,r)=\omega(x)+r G(x,r)$, and $H_r=H(\cdot,r):S^{n-1}\rightarrow\R$.

By \cite[Corollary 2.8]{milnor}, there exists $a_0<0$ such that any $a_0<a<0$ is a regular value of $\omega|S^{n-1}$.
Hence there exists $r_0>0$ such that $a$ is a regular value of every $H_r$, where $-r_0<r<r_0$.
Then 
$$\{(x,r)\in S^{n-1}\times (-r_0,r_0)\ |\ H(x,r)\leq a\}$$
is an $n$-dimensional manifold with boundary $S^{n-1}\times (-r_0,r_0)\cap H^{-1}(a)$.
By the implicit function theorem, the mapping $(x,r)\mapsto r$ restricted to both above manifolds
is a proper submersion. By Ehresmann's theorem, it is a locally trivial
fibration. Hence if $r$ is sufficiently close to zero then the manifolds
$ S^{n-1}\cap \{\omega\leq a\}=\{x\in S^{n-1}\ |\ H(x,0)\leq a\}$ and 
$S^{n-1}\cap \{ H_r\leq a\}=\{x\in S^{n-1}\ |\ H(x,r)\leq a\}$ are homeomorphic.

 The set $S^{n-1}\cap\{\omega\leq a\}$ is a deformation retract
of $\Omega=S^{n-1}\cap\{\omega<0\}$,
so that these sets have the same homotopy type.

We have $f(rx)=r^d H_r(x)$. Hence $x\in S^{n-1}\cap\{H_r\leq a\}$ if and only if
$rx\in S_r^{n-1}\cap\{f\leq a r^d\}$, 
and the proof is complete.
\end{proof}

\section{Geometric properties of trajectories.} \label{rozdzial_22}
In the beginning of this section we present some results obtained by Kurdyka {\it et al.} \cite{kurdykaetal}, \cite{kurdykaparusinski} in the course of proving Thom's gradient conjecture. In exposition and notation we follow closely these papers.

Let $f:\R^n,0\rightarrow\R,0$ be an analytic function defined in a neighbourhood of the origin,
having a critical point at $0$. The gradient $\nabla f(x)$ splits into its radial component
$\frac{\partial f}{\partial r}(x)\frac{x}{|x|}$ and the spherical one
$\nabla'f(x)=\nabla f(x)-\frac{\partial f}{\partial r}(x)\frac{x}{|x|}$.
We shall denote $\frac{\partial f}{\partial r}$ by $\partial_rf$.

For $\epsilon>0$ define
$W^{\epsilon}=\{x\ |\ f(x)\neq 0\ ,\ \epsilon |\nabla'f|\leq |\partial_r f|\}$.
There exists a finite subset of positive rationals $L(f)\subset\Q^+$ such that for any $\epsilon>0$ and any sequence
$W^{\epsilon}\ni x\rightarrow 0$ there is a subsequence $W^{\epsilon}\ni x'\rightarrow 0$
and $\ell\in L(f)$ such that
$$\frac{|x'|\ \partial_r f(x')}{f(x')}\ \rightarrow \ \ell\ .$$
Elements of $L(f)$ are called {\it characteristic exponents}.

Fix $\ell >0$, not necessarily in $L(f)$, and consider $F=f/|x|^{\ell}$ defined in the complement of the origin.
We say that $a\in\R$ is {\it  an asymptotic critical value} of $F$ at the origin if there exists a sequence $x\rightarrow 0$, $x\neq 0$,
such that
$$  |x|\cdot \left|\nabla   F(x)\right|\ \rightarrow 0\ \ ,\ \ 
F(x)=\frac{f(x)}{|x|^\ell}\ \rightarrow\ a\ . $$
The set of asymptotic critical values of $F$ is finite.

The real number $a\neq 0$ is an asymptotic critical value if and only if there exists a sequence
$x\rightarrow 0$, $x\neq 0$, such that
$$ \frac{|\nabla'f(x)|}{|\partial_r f(x)|}\ \rightarrow 0\ \ ,\ \ 
  \frac{f(x)}{|x|^\ell}\ \rightarrow\ a\ . $$

Hence the set
$$L'(f)\ =\ \{(\ell,a)\ |\ \ell\in L(f), a < 0\ \mbox{is an asymptotic critical value of}\ f/|x|^{\ell}\}$$
is a finite subset of $\Q^+\times \R_-$, where $\R_-$ is the set of negative real numbers.

We shall write $T(f)$ for the set of non-trivial trajectories of the gradient flow
$\dot{x}=\nabla f(x)$ converging to the origin. By Section 6 of \cite{kurdykaetal}, for every such a trajectory
$x(t)$, with $x(t)\rightarrow 0$,
there exists a unique pair $(\ell',a')\in L'(f)$ such that $f(x(t))/|x(t)|^{\ell'} \rightarrow a'$.
There is a natural partition of $T(f)$ associated with $L'(f)$. Namely for $(\ell',a')\in L'(f)$,
$$T^{\ell',a'}(f)=\{ x(t)\in T(f)\ |\ f(x(t))/|x(t)|^{\ell'}\rightarrow a'\ \mbox{ as } x(t)\rightarrow 0\}.$$
In the set $\Q^+\times\R_-$ we can introduce the lexicographic order
$$(\ell',a')\leq (\ell,a)\ \ \mbox{if}\ \ \ell'<\ell, \ \mbox{or}\ \ell'=\ell\ \mbox{and}\ a'\leq a.$$
Take $(\ell,a)\in \Q^+\times\R_-\setminus L'(f)$. 
We  shall write 
$$\tilde{T}^{\ell,a}(f)=\bigcup T^{\ell',a'}(f), \mbox{ where } (\ell',a')<(\ell,a) \mbox{ and }(\ell',a')\in L'(f).$$

According to \cite{nowelszafraniec1}, there are $0<-y\ll r\ll 1$ such that each trajectory  $x(t)\in T(f)$ intersects $F_r(y)$ transversally at exactly one point.
Let $\Gamma(f)\subset F_r(y)$ be the union of all those points. By \cite{lojasiewicz} the set $\Gamma(f)$ is closed subset of $F_r(y)$, hence it is compact.
So there is a natural one-to-one correspondence between trajectories in $T(f)$ and 
points in $\Gamma(f)$. The same way one can define the set $\Gamma^{\ell',a'}(f)\subset F_r(y)$  (resp. $\tilde{\Gamma}^{\ell,a}(f)\subset F_r(y)$)
whose points are in one-to-one correspondence with trajectories from $T^{\ell',a'}(f)$ (resp. $\tilde{T}^{\ell,a}(f)$).
In particular, for $(\ell,a)\in \Q^+\times\R_-\setminus L'(f)$ the set
$$\tilde{\Gamma}^{\ell,a}(f)=\bigcup \Gamma^{\ell',a'}(f), \mbox{ where } (\ell',a')<(\ell,a) \mbox{ and }(\ell',a')\in L'(f),$$
is a subset of $\Gamma(f)$.

By \cite[Theorem 12]{nowelszafraniec1}, \cite[Theorem 6]{dzedzejszafraniec}  and Lemma \ref{plik_6_2_1} we have

\begin{theorem} \label{plik_2_1}  
If $0<-y\ll r\ll 1$ then the inclusion $\Gamma(f)\subset F_r(y)$ induces an isomorphism
$$\bar{H}^*(\Gamma(f))\ \simeq\ H^*(F_r(y)),$$
where $\bar{H}^*(\cdot)$ is the \v{C}ech-Alexander-Spanier cohomology group
and $H^*(\cdot)$ is the singular  cohomology group..
In particular $\Gamma(f)$ has the same (finite) number of connected components as $F_r(y)$.

Moreover, for every $(\ell,a)\in\Q^+\times\R_-\setminus L'(f)$ the set
$\tilde\Gamma^{\ell,a}(f)$ is a compact subset of
$F^{\ell,a}(y)$. The inclusion
 induces an isomorphism
$$\bar{H}^*(\tilde\Gamma^{\ell,a}(f))\ \simeq\ H^*(F^{\ell,a}(y)).$$

\end{theorem}

\section{Cardinality of $T(f)$}\label{rozdzial_cardinality}

The cardinality of the set $T(f)$ is obviously the same as that of $\Gamma(f)$. In this section we shall present  simple topological
conditions which imply that $\Gamma(f)$ and  $T(f)$ are  infinite sets.

We shall write $S(f)$ for the stable set of the origin, which is the union of all orbits of the solutions
that converge to the origin.

\begin{rem}\label{plik_2_1_1}
If $\Gamma(f)$ is infinite then it contains at least one  compact and infinite connected component,
which is obviously not a zero-dimensional space.
If that is the case then the Menger-Urysohn dimension as well as the \v{C}ech-Lebesgue covering dimension of this component is at least one (see \cite{engelking}), sets $\Gamma(f)$
and $T(f)$ have the cardinality of the continuum, and the dimension of the stable set
$S(f)$ is at least two.
\end{rem}

By Lemma \ref{plik_6_1}, Corollary \ref{plik_6_2} and Theorem \ref{plik_2_1} we get

\begin{cor}\label{plik_2_3} There is an isomorphism
$\bar{H}^*(\Gamma(f))\ \simeq\ H^*(S_r)$. 
In particular $\Gamma(f)$ has the same (finite) number of connected components as $S_r$.
If there exists $i\geq 1$ such that $H^i(S_r)\neq 0$ then $T(f)$ is infinite.
So, if $S_r\neq \emptyset$ and the Euler-Poincar\'e characteristic $\chi(S_r)\leq 0$,
then $T(f)$ is infinite.

Moreover, for every $(\ell,a)\in\Q^+\times\R_-\setminus L'(f)$, if $0< r\ll 1$ then
$$\bar{H}^*(\tilde\Gamma^{\ell,a}(f))\ \simeq\ H^*(S_r^{\ell,a}).$$

\end{cor}

\begin{ex}\label{plik_2_3_1} The polynomial $f(x,y,z)=x^3+x^2z-y^2$
is weighted homogeneous. Of course $S_r\neq\emptyset$.
By \cite [p.245]{szafraniec12}, the Euler-Poincar\'e characteristic $\chi(S_r^2\cap\{f\geq 0\})=2$.
By the Alexander duality theorem we have $\chi(S_r)=0$. Hence the set $T(f)$ is infinite.
\end{ex}

\begin{prop}\label{plik_7_1}  If $0<-a\ll 1$ then $\bar{H}^*(\tilde\Gamma^{d,a}(f))\simeq H^*(\Omega)$.
If $H^i(\Omega)\neq 0$ for some $i\geq 1$ then $T(f)$ is infinite.
\end{prop}
\begin{proof} As $L'(f)$ is finite, if $0<-a\ll  1$ then $(d,a)\not\in L'(f)$.
By  Corollary \ref{plik_2_3}  and Lemma \ref{plik_5_1}, if $0<r\ll -a$ then we have
$$\bar{H}^*(\tilde\Gamma^{d,a}(f))\ \simeq\ H^*(S_r^{d,a})\ \simeq\ H^*(\Omega) .$$
In particular, if $H^i(\Omega)\neq 0$ for some $i\geq 1$ then 
 $\tilde\Gamma^{d,a}(f)$ is infinite. 
Hence $\tilde{T}^{d,a}(f)$, as well as $T(f)$, is infinite.
\end{proof}

\begin{ex} Let $f(x,y,z)=z(x^2+y^2)+x^2y^2 z-z^4$. It is easy to see that $S_r=S_r^2\cap\{f<0\}$
is homeomorphic to a union of two disjoint 2-discs, so that $H^i(S_r)=0$ for $i\geq 1$.
As $\omega=z(x^2+y^2)$, then $\Omega$ is homeomorphic to $S^1\times(0,1)$,
and so $H^1(\Omega)\neq 0$. Hence $T(f)$ is infinite.
\end{ex}

\begin{cor}\label{plik_7_2} 
If    $\Omega\neq \emptyset$  and  the  Euler-Poincar\'e characteristic  $\chi(\Omega)\leq 0$,  then $T(f)$ is infinite.
\end{cor}

\begin{rem}\label{plik_7_1_1} If $\omega$ is a quadratic form which may be reduced
to the diagonal form 
$-x_1^2-\cdots-x_{i+1}^2+x_{i+2}^2+\cdots+x_j^2,$
where $i\geq 1$, then
the dimension of the stable manifold at the origin is at least two.
Hence $T(f)$ is infinite.

\end{rem}

Investigating the gradient flow in polar coordinates and applying arguments presented by Moussu in \cite[p.449]{moussu}
the reader may also prove the next proposition. (As its proof would require to introduce other techniques ,
so we omit it here.)

\begin{prop}\label{plik_7_3}
Suppose that  there exists a non-degenerate critical point of $\omega|\Omega$
which is not a local minimum. Then $T(f)$ is infinite. 

In particular, if there exists a non-degenerate local maximum of $\omega|\Omega$
then the interior of the stable set of the origin is non-empty.
\end{prop}

\begin{ex}\label{ex_7_4} Let $f(x,y)=x^3+3xy^2+x^2 y^2$, so that $\omega=x^3+3xy^2$. 
It is easy to see that $\omega|S^1$ has a non-degenerate  local maximum at $(-1,0)\in\Omega$.
Then the interior of the stable set of the origin is non-empty.
In particular $T(f)$ is infinite.
\end{ex}

The next theorem is the main result of this section.

\begin{theorem}\label{plik_8_1}
Suppose that $f:\R^n,0\rightarrow\R,0$ is an analytic function having
a critical point at the origin

If $\rank\,H^0(S_r)<\rank\, H^0(\Omega)$, i.e. the number of connected components of $S_r$ 
is smaller than the number of connected components of $\Omega$, then
the set of trajectories of the gradient flow $\dot{x}=\nabla f(x)$ converging to the origin
is infinite.
\end{theorem}
\begin{proof}
Suppoose, contrary to our claim, that $T(f)$ if finite. Then $\Gamma(f)$ is finite, and
 for any $(\ell,a)\in \Q^+\times\R_-\setminus L'(f)$ the set $\tilde\Gamma^{\ell,a}(f)$ is finite too. Hence 
$\rank \bar{H}^0(\tilde\Gamma^{\ell,a}(f))$
equals the number of elements in $\tilde\Gamma^{\ell,a}(f)$.

By Lemma \ref{plik_5_1}, there exist $0<r\ll -a\ll 1$ such that $\Omega$ and  $S_r^{d,a}$ have the same homotopy type.
By Corollary \ref{plik_2_3}, the group $H^*(S_r)$ is isomorphic to $\bar{H}^*(\Gamma(f))$.
Hence $\rank H^0(S_r)$$=\rank \bar{H}^0(\Gamma(f))$
equals the number of elements in $\Gamma(f)$. Moreover,
$\rank H^0(\Omega)$$=\rank H^0(S_r^{d,a})$$=\rank \bar{H}^0(\tilde\Gamma^{d,a}(f))$
equals the number of elements in $\tilde\Gamma^{d,a}(f)$.

As $\tilde\Gamma^{d,a}(f)\subset \Gamma(f)$, then $\rank H^0(\Omega)\leq\rank H^0(S_r)$,
which contradicts the assumption.
\end{proof}

\begin{theorem}\label{plik_8_2}
If $\chi(S_r)<\chi(\Omega)$ then $T(f)$ is infinite.
\end{theorem}

\begin{proof} By Corollary \ref{plik_2_3} and Proposition \ref{plik_7_1}, it is enough to consider the case
where all cohomology groups $H^i(S_r)$, $H^i(\Omega)$, where $i\geq 1$, are trivial.

If that is the case then
$\rank H^0(S_r)=\chi(S_r)<\chi(\Omega)=\rank H^0(\Omega).$
By Theorem \ref{plik_8_1}, the set $T(f)$ is infinite.
\end{proof}

\begin{ex}\label{plik_8_3} Let $f(x,y)=x^3-y^2$, so that $\omega=-y^2$.
Then $\Omega=\{ (x,y)\in S^1\ |\ -y^2<0    \}= S^1\setminus \{(\pm 1, 0)\}$.
Obviously $\Omega$ has two connected components and $H^i(\Omega)=0$ for any $i\geq 1$.
The function $\omega |\Omega$ has exactly two critical (minimum) points at $(0,\pm 1)$,
so one cannot apply Proposition \ref{plik_7_3}.

As $S_r$ is homeomorphic to an interval, then by Theorem \ref{plik_8_1} the set $T(f)$ is infinite.
\end{ex}

\begin{ex}\label{plik_8_4} Let $f(x,y,z)=xyz-z^4$, so that $\omega=xyz$.
It is easy to see that $\Omega$ is homeomorphic to a disjoint union of four discs,
and $S_r$ is homeomorphic to a disjoint union of two discs. By Theorem \ref{plik_8_1} the set
$T(f)$ is infinite.
\end{ex}

\begin{ex}\label{plik_8_5} Let $f(x,y,z)=xyz+x^4y-2y^4z+3xz^4$,
so that $f$ has an isolated critical point at the origin and $\omega=xyz$.
Applying Andrzej {\L}\c ecki computer program (see \cite{leckiszafraniec1}) we have verified that
the local topological degree of the mapping
$$\R^3,0\,\ni (x,y,z)\ \mapsto\ -\nabla f(x,y,z)\, \in\R^3,0$$
equals zero. By \cite{khimshiashvili1}, \cite{khimshiashvili2}, the Euler-Poincar\'e characteristic $\chi(S_r^2\cap\{f\geq 0\})=1-0=1$.
By the Alexander duality theorem $\chi(S_r)=1$. By Theorem \ref{plik_8_2} the set $T(f)$ is infinite.
\end{ex}

\section{Interior of the stable set.}\label{rozdzial_2}

In this section we shall present simple topological conditions which imply that the interior
of the stable set $S(f)$ has a non-empty interior

The set $\Omega$ is semi-algebraic, hence $\operatorname{rank}H^{n-2}(\Omega)<\infty$.
By Theorem \ref{plik_2_1} and Proposition \ref{plik_7_1}, if $0<-a\ll 1$ then $\tilde{\Gamma}^{d,a}(f)$
is compact and \\$\operatorname{rank}\bar{H}^{n-2}(\tilde{\Gamma}^{d,a}(f))<\infty$.

\begin{rem}\label{plik_7_1_1} If $\omega$ is a quadratic form which can be reduced
to the diagonal form 
$-x_1^2-\cdots-x_{i+1}^2+x_{i+2}^2+\cdots+x_j^2,$
where $i\geq 1$, then
$$\bar{H}^*(\tilde{\Gamma}^{d,a}(f))\ \simeq\ H^*(\Omega)\ \simeq H^*(S^i).$$

In that case
$\operatorname{rank}\bar{H}^{n-2}(\tilde{\Gamma}^{d,a}(f))=\operatorname{rank}H^{n-2}(S^i)>0$ if and only if
$\omega$ can be reduced to the diagonal form $-x_1^2-\cdots-x_{n-1}^2$.
\end{rem}

The next theorem is the main result of this section.

\begin{theorem}\label{teor_1}
Suppose that $f:{\mathbb R}^n,0\rightarrow{\mathbb R},0$, $n\geq 2$, is an analytic function defined
in an open neighbourhood of the origin. Suppose that 
$\operatorname{rank}H^{n-2}(S_r)<\operatorname{rank}H^{n-2}(\Omega)$.
Then the stable set of the origin of the gradient flow $\dot{x}=\nabla f(x)$ 
has a non-empty interior.
\end{theorem}
\begin{proof}
By \cite[Lemma 5.10]{milnor}, if $0<-y\ll r\ll 1$ then the Milnor number $F_r(y)$ is homeomorphic to 
an $(n-1)$-dimensional submanifold of $S_r^{n-1}$. 

As $\tilde{\Gamma}^{d,a}(f)\subset\Gamma(f)$ are compact subsets of $F_r(y)$ with
$\operatorname{rank}\bar{H}^{n-2}(\Gamma(f))=\operatorname{rank}H^{n-2}(S_r)<
\operatorname{rank}H^{n-2}(\Omega)=\operatorname{rank}\bar{H}^{n-2}(\tilde{\Gamma}^{d,a}(f))<\infty,$
then by Corollary \ref{cor_1} the set $\Gamma(f)$ has a non-empty interior in $F_r(y)$.

Trajectories of the flow $\dot{x}=\nabla f(x)$ converging to the origin cut transversally $F_r(y)$ at point of $\Gamma(f)$.
Hence the stable set of the origin has a non-empty interior.
\end{proof}

Put $\Omega'=S^{n-1}\cap \{\omega\geq 0\}=S^{n-1}\setminus\Omega$, and $S_r'=S_r^{n-1}\cap \{f\geq 0\}=S_r^{n-1}\setminus S_r$, $0<r\ll 1$.
Sets $\Omega'$, $S_r'$ are compact and semianalytic, hence they are triagulable. By the Alexander duality theorem,
if $S_r'$ and $\Omega'$ are non-empty then
$\operatorname{rank}H_0(S_r')=1+\operatorname{rank}H^{n-2}(S_r)$ and
$\operatorname{rank}H_0(\Omega')=1+\operatorname{rank}H^{n-2}(\Omega)$.

\begin{theorem}\label{theor_2}
Suppose that the set $S_r'$ has less connected components than $\Omega'$.
Then the stable set of the origin of the gradient flow $\dot{x}=\nabla f(x)$ 
has a non-empty interior.
\end{theorem}
\begin{proof}
The set $\Omega'$ is obviously not empty. If $S_r'=\emptyset$ then the origin is a strict local maximum,
and then $\operatorname{int}S(f)\neq\emptyset$. 

Suppose that $S_r'\neq\emptyset$.
Sets $S_r'$, $\Omega'$ are compact, semianalytic. So they are triangulable, and
the number of connected components of $S_r'$ (resp. $\Omega'$) equals the number  
of its path-components which is $\operatorname{rank}H_0(S_r')$
(resp. $\operatorname{rank}H_0(\Omega')$). 

By assumption, 
$\operatorname{rank}H_0(S_r')<
\operatorname{rank}H_0(\Omega')$ and then 
$\operatorname{rank}H^{n-2}(S_r)<\operatorname{rank}H^{n-2}(\Omega)$.
By Theorem \ref{teor_1}, the stable set $S(f)$ has a non-empty interior.
\end{proof}

\begin{ex}\label{plik_8_3_1} Let $f(x,y)=x^3-y^2$ be the same as in Example \ref{plik_8_3}.
Then $\Omega'=\{(-1,0),(1,0)\}$.
As $S_r'$ is homeomorphic to a closed  interval, then by Theorem \ref{theor_2} the interior
of $S(f)$ is non-empty.
\end{ex}

\begin{ex}
Let $f(x,y,z)=-x^2 y^2-z^4+x^5$. Then $\omega=-x^2 y^2-z^4$ and $\Omega'$
consists of four points. It is easy to see that $S_r'$ is homeomorphic to
a disjoint union of a closed disc and two points. By Theorem \ref{theor_2}
the interior of $S(f)$ is non-empty.
\end{ex}

\section{Right-equivalent functions}\label{rozdzial_3}
Let $g:\R^n,0\rightarrow\R,0$ be an analytic function which is right-equivalent to $f$, i.e. there exists
a $C^\infty$-diffeomorphism $\phi:{\mathbb R}^n,0\rightarrow {\mathbb R}^n,0$ defined
in an open neighbourhood of the origin such that $g=f\circ\phi$. Then in particular the derivative
$D\phi(0):{\mathbb R}^n\rightarrow{\mathbb R}^n$ is a linear isomorphism.

Let $\theta$ be the initial homogeneous form associated with $g$, 
let $\Theta=S^{n-1}\cap\{\theta<0\}$, 
and let $\Theta'=S^{n-1}\cap\{\theta\geq 0\}$.
It is easy to see that $\theta=\omega\circ D\phi(0)$. Hence sets 
$\Omega$ and $\Theta$,  as well as $\Omega'$ and $\Theta'$, are homeomorphic.
Then $H^0(\Omega)\simeq H^0(\Theta)$ and
$H_0(\Omega')\simeq H_0(\Theta')$.

 Both $f$ and $g$ are analytic, hence there exists small $r_0>0$ 
such that for each $0<r\leq r_0$ the number of connected components
of $S_r'$ equals the number of connected components of $(B_r^n\setminus\{0\})\cap\{f\geq 0\}$,
and the the number of connected components of $S_r^{n-1}\cap\{g\geq 0\}$ equals
the number of connected components of $(B_r^n\setminus\{0\})\cap\{g\geq 0\}$.
As  $g=f\circ \phi$ then
$(B_{r}^n\setminus\{0\})\cap\{g\geq 0\}$ is homeomorphic to 
$(\phi(B_{r}^n)\setminus\{0\})\cap\{f\geq 0\}$.

There exist $0<r_3<r_2<r_1<r_0$ such that
$\phi(B_{r_3}^n)\subset
B_{r_2}^n\subset \phi(B_{r_1}^n)
\subset B_{r_0}^n$.

The inclusion $(B_{r_3}^n\setminus\{0\})\cap\{g\geq 0\}\subset (B_{r_1}^n\setminus\{0\})\cap\{g\geq 0\}$
is a homotopy equivalence. Hence inclusions
$$(\phi(B_{r_3}^n)\setminus\{0\})\cap\{f\geq 0\}\subset
 (\phi(B_{r_1}^n)\setminus\{0\})\cap\{f\geq 0\},$$
$$(B_{r_2}^n\setminus\{0\})\cap\{f\geq 0\}
\subset ( B_{r_0}^n\setminus\{0\})\cap\{f\geq 0\}$$
are homotopy equivalencies, and then in particular sets $(B_{r_1}^n\setminus\{0\})\cap\{g\geq 0\}$,
$(\phi(B_{r_1}^n)\setminus\{0\})\cap\{f\geq 0\}$ and
$( B_{r_0}^n\setminus\{0\})\cap\{f\geq 0\}$ have the same number of connected components.

Hence sets $S_r'$ and $S_r^{n-1}\cap\{g\geq 0\}$  have the same number of connected components too. By similar arguments, the sets $S_r$ and $S_r^{n-1}\cap\{g<0\}$
have the same number of connected components too.
By Theorems \ref{plik_8_1}, \ref{theor_2} we get

\begin{theorem}\label{cor_4_1}
Suppose that an analytic function $g:\R^n,0\rightarrow \R,0$ is right-equivalent to $f$.
If $S_r$ has less connected components than $\Omega$ then $T(g)$ is infinite.
If $S_r'$ has less connected components than $\Omega'$ then $S(g)$ has a non-empty interior.
\end{theorem}

The next example demonstrates that the assumptions of Theorem \ref{cor_4_1} are significant.

\begin{ex}
Let $f(x,y)=x^3+3xy^2$, so that $S_r'$ and $\Omega'$ are homeomorphic. The same way as in Example \ref{ex_7_4} one can show that the interior
of $S(f)$ is non-empty. The function $g(x,y)=f(\sqrt{3}x,y)=3\sqrt{3}(x^3+xy^2)$ is right-equivalent
to $f$. Applying the polar coordinates one can show that $S(g)$ consists of a single trajectory,
so that its interior is empty.
\end{ex}

In the case where $g$ has an algebraically isolated critical point at the origin, one can compute its Milnor number
$\mu(g)=\dim_{\R}\R[[x_1,\ldots,x_n]]/\langle\partial g\rangle$, where $\langle \partial g\rangle$
is the ideal in $\R[[x_1,\ldots,x_n]]$ generated by $\partial g/\partial x_1,\ldots,\partial g/\partial x_n$ (see \cite{milnor}).

\begin{theorem}\label{ostatnietwierdzenie}
Let $g:\R^n,0\rightarrow\R,0$ be an analytic function having an algebraically isolated critical point
at the origin.
Suppose that $\mu(g)$ is even, and $\theta$ is a quadratic form which can be reduced to the diagonal form $-x_2^2-\cdots-x_n^2$.
Then the interior of $S(g)$ is non-empty.
\end{theorem}

\begin{proof} Applying standard methods of the singularities theory (see \cite{brockerlander})
one can show that $g$ is right-equivalent to $f=x_1^k -x_2^2-\ldots-x_n^2$, where $k=\mu(g)+1$. Then
$S_r'$ is homeomorphic to a closed $(n-2)$-dimensional closed ball and $\Omega'$ consists of two points. 
By Theorem \ref{cor_4_1}, the set $S(g)$ has a non-empty interior.
\end{proof}

\begin{ex} Let $g(x,y,z,w)= x^5+z^5+2zw-x^2-y^2-z^2-w^2-2xyz-y^2z^2$. In this case $\mu(g)=4$, and
$\theta$ can be reduced to the diagonal form
$-y^2-z^2-w^2$.  By Theorem \ref{ostatnietwierdzenie}, the set $S(g)$ has a non-empty interior.
\end{ex}

Zbigniew SZAFRANIEC\\
Institute of Mathematics, University of Gda\'nsk\\
80-952 Gda\'nsk, Wita Stwosza 57, Poland\\
Zbigniew.Szafraniec@mat.ug.edu.pl\\

%%%%%%%%%%%%%%%%%%%%%%%%%%%%%%%%%%%%%%%%%%%%%%%%%%%%%%%%%%%%%%%%%%%%%%%%%%%%%%%%%%%%%%%%%%%%%%%
\end{document}